\documentclass[graybox]{svmult}

\usepackage{mathptmx}       
\usepackage{helvet}         
\usepackage{courier}        
\usepackage{makeidx}         
\usepackage{graphicx}        
\usepackage{multicol}        
\usepackage[bottom]{footmisc}

\usepackage{amsfonts}
\usepackage{amsmath}
\usepackage[italian,english]{babel}

\usepackage{amsthm}
\usepackage{fancyhdr}
\usepackage{amssymb}
\usepackage{graphics}
\usepackage{verbatim}
\usepackage{xcolor}
\usepackage{amsmath,amssymb}

\makeindex             

\newcommand{\field}[1]{\mathbb{#1}}

\newcommand{\Z}{\field{Z}}

\newcommand{\design}{{\mathcal D}}
\newcommand{\fraction}{{\mathcal F}}

\DeclareMathOperator{\xRe}{Re}
\DeclareMathOperator{\somma}{Su}

\begin{document}

\title*{Unions of Orthogonal Arrays and their aberrations via {H}ilbert bases}
\author{Roberto Fontana and Fabio Rapallo}
\institute{Roberto Fontana \at Department of Mathematical Sciences, Politecnico di Torino, corso Duca degli Abruzzi 24, 10124 Torino, Italy, \email{roberto.fontana@polito.it}
\and Fabio Rapallo \at Department of Sciences and Technological Innovation, Universit\`a del Piemonte Orientale, viale Teresa Michel 11, 15121 Alessandria, Italy, \email{fabio.rapallo@uniupo.it}}
%
%
\maketitle

\abstract*{[TO DO]}

\abstract{We generate all the Orthogonal Arrays (OAs) of a given size $n$ and strength $t$ as the union of a collection of OAs which belong to an inclusion-minimal set of OAs. We derive a formula for computing the (Generalized) Word Length Pattern of a union of OAs that makes use of their polynomial counting functions. In this way the best OAs according to the Generalized Minimum Aberration criterion can be found by simply exploring a relatively small set of counting functions. The classes of OAs with $5$ binary factors, strength $2$, and sizes $16$ and $20$ are fully described.}

\section{Introduction}
\label{sec:1}

Design of Experiments plays a central role in several fields of applied Statistics, from Biology to Engineering, from Computer Science to Economics. The need of efficient experimental designs has led to the definition of several criteria for the choice of the design points. All such criteria aim to produce the best estimates of the relevant parameters for a given sample size. Here we limit our attention to fractional factorial designs together with the Generalized Minimum Aberration (GMA) criterion.

In the framework of factorial experiments, Generalized Word-Length Pattern (GWLP) is an important tool for
comparing fractional factorial designs. First introduced for regular fractions, GWLP has been generalized for non-regular multilevel designs by Xu and Wu \cite{xu|wu:01}. Since the GWLP does not depend on the coding of the factor levels, Pistone
and Rogantin \cite{pistone|rogantin:08} use the complex coding of the factor levels to express the basis of the
polynomial complex functions over a design, and in particular of the counting function. Using this coding, the coefficients of the
counting function are closely related with the aberrations and the GWLP. Moreover, the coefficients of the counting function can be expressed in terms of the counts of the levels appearing in each simple or interaction term. As general references for
GWLP and its properties, the reader can refer to, e.g., \cite{hedayat2012orthogonal}, \cite{dey2009fractional} and \cite{mukerjee2007modern}.

In practice, the GWLP is used to discriminate among different designs through the Generalized Minimum Aberration (GMA) criterion: given two designs ${\mathcal F}_1$ and ${\mathcal F}_2$ with $m$ factors, the corresponding GWLPs are two vectors
\[
A_{\fraction_i}=\left(A_0(\fraction_i)=1, A_1(\fraction_i), \ldots , A_m(\fraction_i)\right) \qquad i=1,2 \, .
\]
The GMA criterion consists in the sequential minimization of such GWLPs: $\fraction_1$ is better than $\fraction_2$ if there exists $j$ such that $A_0(\fraction_1) = A_0(\fraction_2), \ldots, A_j(\fraction_1) = A_j(\fraction_2)$ and $A_{j+1}(\fraction_1) < A_{j+1}(\fraction_2)$. The GMA criterion is usually applied to Orthogonal Arrays (OA), see \cite{hedayat2012orthogonal}.

In this work we use results from Combinatorics and Algebraic Geometry to ease the computation of the GWLP. The connection between the GWLP and the geometric structure of the design points is studied in \cite{fontana|rapallo|rogantin:16}, but we adopt here a different point of view. In particular, we show that the set of all OAs with given strength form are the points with integer entries of a cone defined through linear constraints. This allows us to write each OA as the union of elements of the Hilbert basis of the cone. Moreover, we show that the GWLP of the union of two or more fractions can be computed from the counting functions of such fractions. The computation of the Hilbert basis is done through combinatorial algorithms and its complexity increases fast with the number of factors and the number of factor levels. Thus, we illustrate explicit computations for relatively small designs. Nevertheless, the theory presented here can have also a theoretical interest and may be the basis of further developments.


\section{Fractions, counting functions and aberration}
\label{sec:alg}

In this section, for ease in reference, we present some relevant results of the algebraic theory of Orthogonal Fractional Factorial Designs and we express the aberration of fractional designs using the coefficients of the polynomial counting function. This presentation is based on \cite{fontana2017generalized}. The interested reader can find further information, including the proofs of the propositions, in \cite{fontana2000classification} and \cite{pistone2008indicator}.

\subsection{Fractions of a full factorial design} \label{sec:fr_ff}
Let us consider an experiment which includes $m$ factors $\design_{j}, \; j=1,\ldots,m$. Let us code the $s_j$ levels of the factor $\design_{j}$ by  the $s_j$-th roots of the unity
\[
\design_{j} = \{\omega_0^{(s_j)},\ldots,\omega_{s_j-1}^{(s_j)}\} \, ,
\]
where $\omega_k^{(s_j)}=\exp\left(\sqrt{-1}\:  \frac {2\pi}{s_j} \ k\right)$, $k=0,\ldots,s_j-1, \ j=1,\ldots,m$.

The \emph{full factorial design} with complex coding is $\design = \design_1 \times \cdots \design_j \cdots \times \design_m$. We denote its cardinality by $\# \design$, $\# \design=\prod_{j=1}^m s_j$.

\begin{definition}
A \emph{fraction} $\fraction$ is a multiset $(\fraction_*,f_*)$ whose underlying set of elements $\fraction_*$ is contained in $\design$ and $f_*$ is the multiplicity function $f_*: \fraction_* \rightarrow \mathbb N$ that for each element in $\fraction_*$ gives the number of times it belongs to the multiset $\fraction$.
\end{definition}
We recall that the underlying set of elements $\fraction_*$ is the subset of $\design$ that contains all the elements of $\design$ that appear in $\fraction$ at least once. We denote the number of elements of a fraction $\fraction$ by $\# \fraction$, with  $\# \fraction= \sum_{\zeta \in \fraction_*} f_*(\zeta)$.

In order to use polynomials to represent all the functions defined over $\design$, including multiplicity functions, we define
\begin{itemize}
\item $X_j$, the $j$-th component function, which maps a point $\zeta=(\zeta_1,\ldots,\zeta_m)$ of $\design$ to its $j$-th component,
\[
X_j \colon \design \ni (\zeta_1,\ldots,\zeta_m)\ \longmapsto \ \zeta_j \in \design_j \, .
\]
The function $X_j$ is a \emph{simple term} or, by abuse of terminology, a \emph{factor}.
\item $X^\alpha=X_1^{\alpha_1} \cdot \ldots \cdot X_m^{\alpha_m}$, $\alpha \in L =  \Z_{s_1} \times \cdots \times  \Z_{s_m}$ i.e., the monomial function
\[
X^\alpha : \design \ni (\zeta_1,\ldots,\zeta_m)\ \mapsto \ \zeta_1^{\alpha_1}\cdot \ldots \cdot \zeta_m^{\alpha_m} \, .
\]
The function $X^\alpha$ is an \emph{interaction term}.
\end{itemize}

We observe that $\{X^\alpha: \alpha \in L =  \Z_{s_1} \times \cdots \times  \Z_{s_m}\}$ is a basis of all the complex functions defined over $\design$. We use this basis to represent the counting function of a fraction according to the following definition.

\begin{definition} \label{de:indicator}

The \emph{counting function} $R$ of a fraction $\fraction$ is a complex polynomial defined over $\design$ so that for each $\zeta \in \design$, $R(\zeta)$ equals the number of appearances of $\zeta$ in the fraction. A $0-1$ valued counting function is called an \emph{indicator function} of a single-replicate fraction $\fraction$.
We denote by $c_\alpha$ the coefficients of the representation of $R$  on $\design$ using the monomial basis
$\{X^\alpha, \ \alpha \in L\}$:
\[
R(\zeta) = \sum_{\alpha \in L} c_\alpha X^\alpha(\zeta), \;\zeta\in\design, \;  c_\alpha \in \mathbb C \, .
\]
\end{definition}

%

With Prop.~\ref{pr:bc-alpha} from \cite{pistone2008indicator}, we link the orthogonality of two interaction terms with the coefficients of the polynomial representation of the counting function. We denote by $\overline{z}$ the complex conjugate of the complex number $z$.

\begin{proposition} \label{pr:bc-alpha}
If $\fraction$ is a fraction of a full factorial design $\design$, $R = \sum_{\alpha \in L} c_\alpha X^\alpha$ is its counting function and $[\alpha-\beta]$ is the $m$-tuple made by the componentwise difference in the rings $\Z_{s_j}$,
 $\left(\left[\alpha_1-\beta_1 \right]_{s_1}, \ldots, \left[\alpha_m - \beta_m\right]_{s_m} \right)$, then
\begin{enumerate}
 \item \label{it:balpha}
the coefficients $c_\alpha$ are given by $c_\alpha= \frac 1 {\#\design} \sum_{\zeta \in \fraction} \overline{X^\alpha(\zeta)}$;
 \item \label{it:cent1}
 the term $X^\alpha$ is centered on $\fraction$ i.e., $\frac{1}{\#\fraction} \sum_{\zeta \in \fraction} X^\alpha(\zeta)=0$ if, and only if,
 $c_\alpha=c_{[-\alpha]}=0 $;
 \item \label{it:cent2}
 the terms $X^\alpha$ and $X^\beta$
are orthogonal on $\fraction$ if and only if,
$c_{[\alpha-\beta]}=0 $.
\end{enumerate}
\end{proposition}

We now define projectivity and, in particular, its relationship with Orthogonal Arrays. Given $I=\{i_1,\ldots,i_k\} \subset \{1,\ldots,m\}, i_1<\ldots < i_k$ and $\zeta=(\zeta_1,\ldots,\zeta_m) \in \design$ we define the projection $\pi_I(\zeta)$ as
\[
\pi_I(\zeta)=\zeta_I \equiv (\zeta_{i_1},\ldots,\zeta_{i_k}) \in \design_{i_1} \times \ldots \times \design_{i_k} \, .
\]

\begin{definition}
A fraction $\fraction$ {\em factorially projects} onto the $I$-factors, $I=\{i_1,\ldots,i_k\} \subset \{1,\ldots,m\}$, $i_1<\ldots < i_k$, if the projection $\pi_I(\fraction)$ is
a multiple full factorial design, i.e., the multiset $(\design_{i_1} \times \ldots \times \design_{i_k} , f_*)$ where the multiplicity function $f_*$ is constant over $\design_{i_1} \times \ldots \times \design_{i_k}$.
\end{definition}

\begin{definition}
A fraction $\fraction$ is a {\em (mixed) Orthogonal
Array (OA)} of strength $t$ if it factorially projects onto any $I$-factors with $\#I=t$.
\end{definition}


\begin{proposition} \label{pr:projectivity}
A fraction factorially projects onto the $I$-factors, $I=\{i_1,\ldots,i_k\} \subset \{1,\ldots,m\}, i_1<\ldots < i_k$, if and only if,
all the coefficients of the counting function involving the $I$-factors only are $0$.
\end{proposition}

Prop.~\ref{pr:projectivity} can be immediately stated for mixed orthogonal arrays.
\begin{proposition} \label{pr:projectivity_ort}
A fraction is an OA of strength $t$ if and
only if all the coefficients $c_{\alpha}, \; \alpha \neq 0 \equiv (0,\ldots,0)$ of the counting function up to the order $t$ are $0$.
\end{proposition}

\subsection{GWLP and aberrations}
\label{sec:ab_crit}

Using the polynomial counting function, \cite{cheng2004geometric} provides the following definition of the GWLP $A_\fraction=(A_0(\fraction), \ldots, A_m(\fraction))$ of a fraction $\fraction$ of the full factorial design $\design$.

\begin{definition} \label{wlp}
The Generalized Word-Length Pattern (GWLP) of a fraction $\fraction$ of the full factorial design $\design$ is a the vector $A_\fraction=(A_0(\fraction),A_1(\fraction), \ldots , A_m(\fraction))$, where
\[
A_j(\fraction)= \sum_{|\alpha |_0 =j} a_\alpha \quad j=0,\ldots,m \, ,
\]
\begin{equation} \label{aberration}
a_\alpha = \left(  \frac{ \|c_{\alpha}\|_2 }{c_{0}} \right)^2 \, ,
\end{equation}
$| \alpha |_0$ is the number of non-null elements of $\alpha$, $\| z \|_2$ is the norm of the complex number $z$, and $c_0 := c_{(0,\ldots,0)}={\#\fraction}/{\#\design}$.
\end{definition}

We refer to $a_\alpha$ as the \emph{aberration} of the interaction $X^\alpha$. In Prop.~\ref{pr:abe} we provide a formula to compute $a_\alpha$, and consequently $A_j(\fraction)$, $j=1,\ldots,m$, given a fraction $\fraction \subseteq \design$. Notice that $A_0(\fraction)=1$ for all $\fraction$. Moreover, in the case of binary designs, the coefficients of the counting function are real numbers and therefore the aberrations in Eq.~\eqref{aberration} are simply
\[
a_\alpha = \left(  \frac{ c_{\alpha}}{c_{0}} \right)^2 \, .
\]

Given a fraction $\fraction$ of the full factorial design $\design$, let us consider its counting function $R = \sum_{\alpha \in L} c_\alpha X^\alpha$. From item \ref{it:balpha} of Prop.~\ref{pr:bc-alpha} the coefficients $c_\alpha$ are given by
\[
c_\alpha= \frac 1 {\#\design} \sum_{\zeta \in \fraction} \overline{X^\alpha(\zeta)}
\]
or equivalently
\[
c_\alpha= \frac 1 {\#\design} \sum_{\zeta \in \design} R(\zeta) \overline{X^\alpha(\zeta)} \, .
\]
To make the notation easier we use vectors and matrices and we make the non-restrictive hypothesis that both the runs $\zeta$ of the full factorial design $\design$ and the multi-indexes of $L=\Z_{s_1} \times \cdots \times  \Z_{s_m}$ are considered in lexicographic order.
We obtain
\[
c_\alpha=\frac 1 {\#\design} \overline{X}_\alpha^T Y = \frac 1 {\#\design} Y^T \overline{X}_\alpha \, ,
\]
where $X_\alpha$ is the column vector $\left[ \zeta^\alpha : \zeta \in \design \right]$, $\overline{X}_\alpha$ is the column vector $\left[ \overline{\zeta}^\alpha : \zeta \in \design \right]$ $Y$ is the column vector $\left[ R(\zeta) : \zeta \in \design \right]$ and the exponent $T$ denotes the transpose of a matrix. The square of the norm of a complex number $z$ can be computed as $ z \overline {z}$. It follows that
\[
\| c_\alpha \|_2^2 =  c_\alpha \overline{ c_\alpha}
\]
and therefore we get
\[
(\#\design)^2 \| c_\alpha \|_2^2 = (Y^T \overline{X}_\alpha) \overline{( \overline {X}_\alpha^T Y )} = Y^T \overline {X}_\alpha {X_\alpha}^{T} Y \, .
\]
As in \cite{fontana2011counting}, we refer to $Y$ as the \emph{counting vector} of a fraction.

\subsection{Counting vector and aberrations}

Here we present some properties of the aberrations and some results about the relationships between the aberrations and the counting vector of a fraction. The results are adapted to the complex coding for multilevel factors.

\begin{proposition} \label{pr:abe}
Given a fraction $\fraction$ it holds:
\begin{enumerate}
\item $a_\alpha=({Y^T \overline {X}_\alpha {X_\alpha}^{T} Y})/{(\#\fraction)^2}$;
\item $\somma(\overline {X}_\alpha {X_\alpha}^{T})=0$, $\alpha \neq 0$ where $\somma(A)$ is the sum of all the elements of the matrix $A$;
\item $\sum_{j=0}^{m} A_j(\fraction) = \sum_{\alpha \in L} a_\alpha = ({\#\design \sum_{\zeta \in \design} Y[\zeta]^2})/{(\#\fraction)^2}$;
\item if $Y[\zeta] \in \{0,1\}, \zeta \in \design$ then $\sum_{i=0}^{m} A_i(\fraction)={\#\design}/{\#\fraction}\equiv c_0^{-1}$.
\end{enumerate}
\end{proposition}
\begin{proof}
\begin{enumerate}
	\item From the definition of $a_\alpha$ we get
	\[
	a_\alpha = \left(  \frac{ \|c_{\alpha}\|_2 }{c_{0}} \right)^2=\frac{(1/\#\design)^2 Y^T \overline {X}_\alpha {X_\alpha}^{T} Y}{(\#\fraction/\#\design)^2}=\frac{Y^T \overline {X}_\alpha {X_\alpha}^{T} Y}{(\#\fraction)^2} \, .
	\]
	\item Let us consider the full factorial design $\design$. Its counting vector is $1$, i.e., the column vector with all the components equal to $1$. The coefficients of its counting function are $c_{0}=1$ and $c_{\alpha}=0$ for all $\alpha \neq 0$. We get $a_{\alpha}=0$ for all $\alpha \neq 0$. It follows that the sum of all the elements of the matrix $\overline {X}_\alpha {X_\alpha}^{T}$ is
	\[
	\somma(\overline {X}_\alpha {X_\alpha}^{T})= 1^T \overline {X}_\alpha {X_\alpha}^{T} 1 = (\#\design)^2 a_\alpha = 0, \; \alpha \neq 0 \, .
	\]
	\item The sum of all the terms of the generalized word-length pattern is
	\begin{eqnarray*}
\sum_{j=0}^{m} A_j(\fraction)=\sum_{\alpha \in L} a_\alpha=\sum_{\alpha \in L} \frac{Y^T \overline {X}_\alpha {X_\alpha}^{T} Y}{(\#\fraction)^2}=\\
=\frac{Y^T \sum_{\alpha \in L} (\overline {X}_\alpha {X_\alpha}^{T}) Y}{(\#\fraction)^2}=
\frac{Y^T \overline {X} {X}^{T} Y}{(\#\fraction)^2} =\\
= \frac{\#\design Y^T Y}{(\#\fraction)^2} = \frac{\#\design \sum_{\zeta \in \design} Y(\zeta)^2}{(\sum_{\zeta \in \design} Y(\zeta))^2} \, ,
\end{eqnarray*}
where $X$ is the orthogonal matrix whose columns are $X_\alpha, \alpha \in L$.
\item It follows from item 3. by observing that $Y[\zeta] \in \{0,1\}, \zeta \in \design \Rightarrow Y[\zeta]^2=Y[\zeta]$ and then $\sum_{\zeta \in \design} Y[\zeta]^2=\#\fraction$.
\end{enumerate}
\end{proof}

From items 3. and 4. of Prop.~\ref{pr:abe} we obtain that, for a given size $n$, the total aberration of a single-replicate fraction $\fraction_1$ (with counting vector $Y_1$) will be less than the total aberration of a fraction $\fraction_2$ (with counting vector $Y_2$) that admits replications. In fact, we get
\[
\sum_{j=0}^m A_j(\fraction_1)=\frac{\#\design}{n}, \qquad
\sum_{j=0}^m A_j(\fraction_2)=\frac{\#\design \sum_{\zeta \in \design} Y_2[\zeta]^2}{n^2}
\]
and
\[
\frac{\#\design}{n} \leq \frac{\#\design}{n} \frac{\sum_{\zeta \in \design} Y_2[\zeta]^2}{n}
\]
because, given $n$, $\sum_{\zeta \in \design} Y_2[\zeta]^2 \geq n$.

Now, as in \cite{gromping2014generalized}, let us consider the special case of OAs of size $n$ and strength $t$ (or equivalently with resolution $t+1$), with $m=t+1$ factors. Using the standard notation, we denote this class of OAs by $OA(n, s_1\ldots s_m, m-1)$. We can state the following proposition.

\begin{proposition} \label{prop:ar}
Let $\fraction \in OA(n, s_1 \ldots s_m, m-1)$. Then
\[
A_m(\fraction)=\frac{\#\design \sum_{\zeta \in \design} Y[\zeta]^2 - n^2}{n^2} \, .
\]
If $\fraction$ is a single-replicate OA (i.e. $Y[\zeta] \in \{0,1\}, \zeta \in \design$) then
\[
A_m(\fraction)=\frac{\#\design-n}{n} \, .
\]
\end{proposition}
\begin{proof}
Let us consider $\fraction \in OA(n, s_1 \ldots s_m, m-1)$. Then
\[
A_0(\fraction)=1,A_1(\fraction)= \cdots = A_{m-1}(\fraction)=0 \, .
\]
From item 3. of Prop.~\ref{pr:abe} we get
\begin{multline*}
A_m(\fraction)=\sum_{j=0}^m A_j(\fraction) - \sum_{j=0}^{m-1} A_j(\fraction)= \\
= \frac{\#\design \sum_{\zeta \in \design} Y[\zeta]^2}{(\#\fraction)^2}-1=  \frac{\#\design \sum_{\zeta \in \design} Y[\zeta]^2-(\#\fraction)^2}{(\#\fraction)^2} \, .
\end{multline*}
In the special case $Y[\zeta] \in \{0,1\}, \zeta \in \design$ we get
\[
A_m(\fraction)=\frac{\#\design - \#\fraction}{\#\fraction} \, .
\]
\end{proof}
We obtain a lower bound for $A_m(\fraction)$ as in Theorem 5 of \cite{gromping2014generalized}.
\begin{proposition} \label{pr:LB}
Let $\fraction \in OA(n, s_1 \ldots s_m, m-1)$. Then
\[
A_m(\fraction)\geq \frac{r(\#\design-r)}{n^2} \, ,
\]
where $q$ and $r$ are the quotient and the remainder when $n$ is divided by $\#\design$, $n=q\#\design+r$ (and $q=0$ when $n < \#\design$).
\end{proposition}
\begin{proof}
From Prop.~\ref{prop:ar} we know that
\[
A_m(\fraction)=\frac{\#\design \sum_{\zeta \in \design} Y[\zeta]^2 - n^2}{n^2} \, .
\]
If we divide $n$ by $\#\design$ we can write $n=q\#\design+r$. The counting vector $\tilde{Y}$ that minimizes $\sum_{\zeta \in \design} Y[\zeta]^2$  must be defined as
\[
\tilde{Y}[\zeta] =
\begin{cases}
q+1 \;\; \text{if } \zeta \in B_r \\
q   \; \; \; \; \; \; \; \; \; \text{if } \zeta \in \design - B_r
\end{cases}
\]
where $B_r$ is any subset of $\design$ with $r$ points.
We obtain
\[
\sum_{\zeta \in \design} \tilde{Y}[\zeta]^2=\#\design q^2 + 2rq +r \, .
\]
It follows that
\[
A_m(\fraction) \geq \frac{\#\design (\#\design q^2 + 2rq +r) - (q\#\design+r)^2}{(q\#\design+r)^2} \, .
\]
By simple algebra we obtain
\[
A_m(\fraction) \geq \frac{r(\#\design - r)}{(\#\fraction)^2} \, .
\]
\end{proof}
When we consider $m>t+1$ factors a lower bound for $A_{t+1}(\fraction)$ can be obtained by summing up all the lower bounds that are obtained using Prop.~\ref{pr:LB} for all the $\binom{m}{t+1}$ subsets of $t+1$ factors of $\design_1,\ldots,\design_m$.

\section{The counting function of the union of fractions}
\label{sec:3}

In this section we analyze the behavior of the aberrations (and thus of the GWLP) of a fraction obtained by merging two or more fractions. In particular we focus on OAs which can be expressed as the union of other OAs.

First, it is worth noting that given a fraction ${\mathcal F}$ with counting function $R(\zeta)$, we can consider a fraction $\nu{\mathcal F}$ obtained by replicating $\nu$ times each design point of ${\mathcal F}$. In such a case, it is immediate to check that the counting function of $\nu{\mathcal F}$ is simply $\nu R(\zeta)$, and therefore all aberrations remain unchanged:
\[
a_\alpha^{(vR)}=a_\alpha^{(R)}, \quad \mbox{ for all } \alpha \in L \, .
\]

In the following proposition we consider the union of $k$ fractions, $k\geq 2$.

\begin{proposition} \label{union:prop}
Let us consider fractions ${\mathcal F}_1, \ldots, {\mathcal F}_k$ with $n_1, \ldots, n_k$ design points, respectively.
Let us denote by $R_i=\sum_{\alpha \in L} c_\alpha^{(i)}$ the counting function of $\fraction_i$, $i=1,\ldots,k$, by ${\mathcal F}$ the union ${\mathcal F} = {\mathcal F}_1 \cup \cdots \cup {\mathcal F}_k$, by $R=\sum_{i=1}^k R_i$ the counting function of $\fraction$ and by $n$ the size of $\fraction$, $n=n_1 + \ldots + n_k$.

The $j$-th element of the GWLP of $\fraction$ is
\begin{equation} \label{WLP-1}
A_j (\fraction) = \sum_{i=1}^k \frac {n_i^2} {n^2} A_i(\fraction_i) + 2 \frac {(\#\design)^2}{n^2} \sum_{i_1<i_2}  \sum_{|\alpha|_0=j} \xRe( c_\alpha^{(i_1)} \overline{c}_\alpha^{(i_2)}), \; \; j=0,\ldots,m \, .
\end{equation}
\end{proposition}
\begin{proof}
Let us consider $k=2$, i.e. $\fraction = \fraction_1 \cup \fraction_2$.
The aberration $a_\alpha^{(R)}$ is
\[
a_\alpha^{(R)} = \frac{(\|c_\alpha^{(1)}+c_\alpha^{(2)}\|_2)^2}{(c_0^{(1)}+c_0^{(2)})^2} \, .
\]
We obtain
\begin{eqnarray*}
(\|c_\alpha^{(1)}+c_\alpha^{(2)}\|_2)^2=(\|c_\alpha^{(1)}\|_2)^2+(\|c_\alpha^{(2)}\|_2)^2+2\xRe(c_\alpha^{(1)} \overline{c}_\alpha^{(2)})=\\
=(\frac{n_1}{\#\design})^2 a_\alpha^{(1)} + (\frac{n_2}{\#\design})^2 a_\alpha^{(2)} + 2\xRe(c_\alpha^{(1)} \overline{c}_\alpha^{(2)})=\\
=\frac{1}{(\#\design)^2} \left( n_1^2 a_\alpha^{(1)} +  n_2^2 a_\alpha^{(2)} +  2(\#\design)^2 \xRe(c_\alpha^{(1)} \overline{c}_\alpha^{(2)}) \right)
\end{eqnarray*}
where $a_\alpha^{(i)}$ refers to $\fraction_i$, i=1,2.
We also obtain
\[
(c_0^{(1)}+c_0^{(2)})^2=(\frac{n_1}{\#\design}+\frac{n_2}{\#\design})^2=\frac{n^2}{(\#\design)^2} \, .
\]
It follows
\[
a_\alpha^{(R)}=\frac{1}{n^2} \left( n_1^2 a_\alpha^{(1)} +  n_2^2 a_\alpha^{(2)} +  2(\#\design)^2 \xRe(c_\alpha^{(1)} \overline{c}_\alpha^{(2)}) \right)
\]
and
\begin{multline*}
A_j(\fraction)=\sum_{|\alpha|_0=j} a_\alpha^{(R)}= \\
= \left(\frac{n_1}{n}\right)^2 A_j(\fraction_1)+ \left(\frac{n_2}{n}\right)^2 A_j(\fraction_2)+
2\left(\frac{\#\design}{n}\right)^2 \sum_{|\alpha|_0=j} \xRe(c_\alpha^{(1)} \overline{c}_\alpha^{(2)})
\end{multline*}
for $j=0,1,\ldots,m$.

The generalization of this formula to the case $k > 2$ is straightforward.
\end{proof}

In case of two-level designs, $c_\alpha \in \field{R}$ and thus Eq.~\ref{WLP-1} becomes
\[
A_j (\fraction) = \sum_{i=1}^k \frac {n_i^2} {n^2} A_i(\fraction_i) + 2 \frac {(\#\design)^2}{n^2} \sum_{i_1<i_2}  \sum_{|\alpha|_0=j} c_\alpha^{(i_1)} c_\alpha^{(i_2)}, \; \; j=0,\ldots,m \, .
\]
The term $\sum_{|\alpha|_0=j} c_\alpha^{(i_1)} c_\alpha^{(i_2)}$ can be viewed as a kind of covariance between the coefficients of order $j$ of the two counting functions $R_{i_1}$ and $R_{i_2}$. 

To illustrate the use of Prop.~\ref{union:prop} on a very small example, let us consider the two regular fractions of the $2^3$ design, whose union is the full-factorial:
\[
{\mathcal F}_1 = {X_1X_2X_3=-1} \qquad R_1= \frac 1 2 (1 - X_1X_2X_3) \, ;
\]
\[
{\mathcal F}_2 = {X_1X_2X_3=+1} \qquad R_1= \frac 1 2 (1 + X_1X_2X_3) \, .
\]
In this case we have
\begin{eqnarray*}
A_0(\fraction_1)=1, A_1(\fraction_1)=A_2(\fraction_1)=0, A_3(\fraction_1)=1 \, ; \\
A_0(\fraction_2)=1, A_1(\fraction_2)=A_2(\fraction_2)=0, A_3(\fraction_2)=1 \, .
\end{eqnarray*}
As expected we obtain $A_0(\fraction)=1, A_1(\fraction)=A_2(\fraction)=0$ and
\[
A_3(\fraction)=\left(\frac{4}{8}\right)^2 A_3(\fraction_1) + \left(\frac{4}{8}\right)^2 A_3(\fraction_2) + 2 \left(\frac{8}{4}\right)^2 c_{111}^{(1)} c_{111}^{(2)} =0
\]
because $c_{111}^{(1)}=-1/2$ and $c_{111}^{(2)}=1/2$.

\section{The Hilbert basis for Orthogonal Arrays}

In this section we define the set $OA(\bullet,{\mathcal D},t)$ of all the OAs with strength $t$ of the full design $\mathcal D$ and we study its combinatorial and geometric properties. With respect to the standard notation, we allow the cardinality to vary, because our study will concern the union of two or more OAs, and thus we use the symbol $\bullet$ in place of the cardinality of the fraction. In the case of binary designs, this set has already been considered in \cite{carlini|pistone:07}, where the reader can find also a simple and comprehensive summary of the basic definitions from Combinatorics used here. The generalization to mixed-level designs can be found in \cite{fontana2013algebraic}.

As a preliminary remark, notice that to the set $OA(\bullet,{\mathcal D},t)$ can be associated in a natural way the set of the corresponding counting functions. With as slight abuse of notation, we use the same notation for both these sets.

\begin{lemma}
The set $OA(\bullet,{\mathcal D},t)$ can be written in the form
\begin{equation} \label{latticecone}
OA(\bullet,{\mathcal D},t) = C \cap {\mathbb N}^{\#{\mathcal D}}
\end{equation}
where $C$ is a polyhedral cone in ${\mathbb R}^{\#\mathcal D}$.
\end{lemma}
\begin{proof}
Recall that a subset of ${\mathbb R}^k$ is a cone if for all $x,y \in C$ and for all $\lambda, \mu \in {\mathbb R}$ we have $\lambda x + \mu y \in C$, and it is a polyhedral cone if in addition it can be written in the form
\begin{equation} \label{latticecone2}
C = \left\{x \in {\mathbb R}^k  \ : \ Ax \geq 0 \right\} \, .
\end{equation}
In this setting it is enough to define the matrix $A$ in such a way all the $t$-marginals of $x$ are constant (i.e., the difference of any two elements in a $t$-marginal is equal to $0$).
\end{proof}

In Combinatorics, objects like $OA(\bullet,{\mathcal D},t)$ expressed as the lattice points of a cone as in Eq. \eqref{latticecone} are widely studied. See, e.g., Chapter 6 in \cite{miller2006combinatorial} for a general introduction to semigroups, lattice ideals, and Hilbert bases. In this paper, we focus on the notion of Hilbert basis of a lattice, and we specialize its definition.

\begin{definition}
A Hilbert basis of $OA(\bullet,{\mathcal D},t)$ is an inclusion-minimal finite set of Orthogonal Arrays ${\mathcal B}_1, \ldots , {\mathcal B}_r$ such that each Orthogonal Array ${\mathcal F} \in OA(\bullet,{\mathcal D},t)$ is
\[
{\mathcal F} = c_1{\mathcal B}_1 + \cdots + c_r{\mathcal B}_r
\]
with coefficients $c_1, \ldots, c_r \in {\mathbb N}$.
\end{definition}

Under mild conditions, which are satisfied by $OA(\bullet,{\mathcal D},t)$, the Hilbert basis exists and is unique.

The Hilbert basis of $OA(\bullet,{\mathcal D},t)$ depends on the matrix $A$ in Eq. \eqref{latticecone2}, which in turn depends on the $t$-marginals of the Orthogonal Arrays. Thus, we have a different Hilbert basis for different ${\mathcal D}$ and $t$. From the computational point of view, there are specific algorithms to efficiently compute Hilbert bases. Such algorithms are available by means of specialized software. Currently, two choices are available: {\tt 4ti2}, see \cite {4ti2}, and the more recent package {\tt normaliz}, see \cite{normaliz}. For our purpose, the use of one or the other software is equivalent. In our examples, we have used {\tt 4ti2}, but the use of both these software is very easy. It is enough to input the matrix $A$ defining the polyhedral cone and the software returns the corresponding Hilbert basis.

Using the elements of the Hilbert basis, we can build all Orthogonal Arrays of any given sample size. As noticed in the Introduction, the limitation of our approach is due to the fact the computation of Hilbert bases is very intensive and the computational cost grows very fast when the full design becomes large. Therefore, the computations are limited to relatively small cases, which are to be considered as illustrative examples.

\section{Computations}

We consider OAs of strength $2$ for $5$ factors, each with $2$ levels, $OA(\bullet,2^5,t)$. The Hilbert Basis for this problem contains $26,142$ different elements which can be classified according to their size as reported in Table \ref{tab:hil52}.
\begin{table}
\caption{The elements of the Hilbert basis for $OA(\bullet,2^5,2)$ classified with respect to their sample size.}
\label{tab:hil52}       
%
%
\begin{tabular}{p{2cm}p{2cm}}
\hline\noalign{\smallskip}
size &	N \\
\noalign{\smallskip}\svhline\noalign{\smallskip}
8 &	60 \\
12 &	224 \\
16 &	162 \\
20 &	960 \\
24 &	7680 \\
28 &	8384 \\
32 &	5760 \\
36 &	2912 \\
\noalign{\smallskip}\hline\noalign{\smallskip}
\end{tabular}
\end{table}

First, we focus on the OAs with size equal to $16$. There are $162$ OAs of size $16$ in the Hilbert Basis. The remaining $16$-run OAs can be generated considering all possible unions of two OAs of size equal $8$.  We denote such OAs as $(8+8)$-run OAs. There are $60$ $8$-run OAs and therefore $60+\binom{60}{2}=1,830$ possibly different $(8+8)$-run OAs. We find $1,770$ different $(8+8)$-run OAs. The classification of the $162+1,770=1,932$ OAs of size $16$ according to the values of $A_3(\fraction)$ is reported in Table \ref{tab:16run}.

\begin{table}
\caption{Distribution of $OA(16,2^5,2)$ with respect to $A_3(\fraction)$.}
\label{tab:16run}       
%
%
\begin{tabular}{p{1,5cm}p{1,1cm}p{1,1cm}p{1,1cm}p{1,1cm}p{1,1cm}p{1,1cm}p{1,1cm}p{1,1cm}}
\hline\noalign{\smallskip}
 & \multicolumn{7}{c}{$A_3(\fraction)$} & \\
type &	$0$ &	$0.25$ &	$0.5$ &	$0.75$ &	$1$ &	$1.5$ &	$2$ & Total \\
\noalign{\smallskip}\svhline\noalign{\smallskip}
$16$-run &	$2$ &	$80$ &	$0$ &	$80$ &	$0$ &	$0$ &	$0$  & $162$ \\
$(8+8)$-run &	$10$ &	$0$ &	$240$ &	$0$ &	$1,220$ &	$240$ &	$60$ & $1,770$ \\
\noalign{\smallskip}\hline\noalign{\smallskip}
\end{tabular}
\end{table}

From Table \ref{tab:16run} we immediately see that there are $12$ designs with $A_3(\fraction)=0$. We can choose the best design(s) among these $12$ fractions. We find $2$ OAs of the $16$-run type for which $A_1(\fraction)=A_2(\fraction)=A_3(\fraction)=A_4(\fraction)=0$ and $A_5(\fraction)=1$.

As a second example, we consider OAs with $20$ runs. There are $960$ OAs of size $20$ in the Hilbert Basis. The remaining $20$-run OAs can be generated by considering all possible unions of two OAs, one of size $8$ and one of size $12$. We denote such OAs as $(8+12)$-run OAs. There are $60$ $8$-run OAs and $224$ $12$-run OAs and therefore $60 \cdot 224=13,440$ possibly different $(8+12)$-run OAs. We find $9,792$ different $(8+12)$-run OAs. The classification of the $960+9,792=10,752$ OAs of size $20$ according to the values of $A_3(\fraction)$ is reported in Table \ref{tab:20run}.

\begin{table}
\caption{Distribution of $OA(20,2^5,2)$ with respect to $A_3(\fraction)$.}
\label{tab:20run}       
%
%
\begin{tabular}{p{1,5cm}p{1,1cm}p{1,1cm}p{1,1cm}p{1,1cm}}
\hline\noalign{\smallskip}
 & \multicolumn{3}{c}{$A_3(\fraction)$} & \\
type &	$0.4$ & $0.72$ & $1.04$ & Total \\
\noalign{\smallskip}\svhline\noalign{\smallskip}
$20$-run &	$480$ &	$0$ &	$480$ & $960$ \\
$(8+12)$-run &	$1,632$ &	$4,800$ &	$3,360$ & $9,792$ \\
\noalign{\smallskip}\hline\noalign{\smallskip}
\end{tabular}
\end{table}

If we proceed as we did for OAs of size $16$, focusing on the $2,112$ OAs with $A_3=0.4$, we find $192$ GMA-optimal OAs. These are of the $(8+12)$-run type and their Word Length Pattern is $A_1(\fraction)=A_2(\fraction)=0,A_3(\fraction)=0.4,A_4(\fraction)=0.2$ and $A_5(\fraction)=0$.

\begin{acknowledgement}
Both authors are partially supported by a INdAM GNAMPA 2017 project.
\end{acknowledgement}

\bibliographystyle{spmpsci}
\bibliography{biblio}

\end{document}